\documentclass{article}
\usepackage{natbib}
\usepackage{amsthm}
\let\cite\citep
\usepackage{geometry}                
\usepackage{rotating}
\usepackage{amsfonts}
\usepackage{graphicx}
\usepackage{amssymb,amsmath}
\usepackage{authblk}
\usepackage[ruled,vlined,linesnumbered]{algorithm2e}
\usepackage{url}
\usepackage{color}
\def\red#1 {\textcolor{red}{#1 }}
\def\blue#1 {\textcolor{blue}{#1 }}
\def\brown#1 {\textcolor{brown}{#1 }}

\newcommand{\inv}{\mathcal I} 
\newcommand{\n}{\mathbf{n} }

\providecommand{\keywords}[1]{\textbf{\textit{Keywords:}} #1}
\providecommand{\subclass}[1]{\textbf{\textit{AMS subclass classification:}} #1}
\def\Z{\mathbb Z}
\def\R{\mathbb R}

\def\s{\sigma}

\newtheorem{theorem}{Theorem}[section]
\newtheorem{proposition}[theorem]{Proposition}
\newtheorem{lemma}[theorem]{Lemma}
\newtheorem{corollary}[theorem]{Corollary}
\theoremstyle{definition}
\newtheorem{definition}[theorem]{Definition}

\usepackage{enumerate}

\usepackage{tikz}
\usetikzlibrary{shadows,matrix,arrows} 

\def\Zn{\Z/n\Z}
\def\N{\mathbb N}
\def\PP{\mathrm P}
\def\D{\mathcal D}

\title{Group-theoretic models of the inversion process in bacterial genomes\thanks{This research was partially funded by Australian Research Council Discovery Grants DP0987302 and DP130100248, and Australian Research Council Future Fellowship FT100100898.}}
\author[1]{Attila Egri-Nagy}
\author[1]{Volker Gebhardt}
\author[2]{Mark M. Tanaka}
\author[1]{Andrew R. Francis}
\affil[1]{
Centre for Research in Mathematics,
University of Western Sydney,
Locked Bag 1797,
Penrith, NSW 2751,
Australia\\
\protect\url{a.egri-nagy,v.gebhardt,a.francis@uws.edu.au}}
\affil[2]{
Evolution \& Ecology Research Centre,
School of Biotechnology and Biomolecular Sciences,
University of New South Wales,
NSW   2052
Australia\\
\protect\url{m.tanaka@unsw.edu.au}}

\date{}

\begin{document}

\maketitle

\begin{abstract} 
The variation in genome arrangements among bacterial taxa is largely due to the process of inversion.
Recent studies indicate that not all inversions are equally probable, suggesting, for instance, that shorter inversions are more frequent than longer, and those that move the terminus of replication are less probable than those that do not.  Current methods for establishing the inversion distance between two bacterial genomes are unable to incorporate such information.  In this paper we suggest a group-theoretic framework that in principle can take these constraints into account.  
In particular, we show that by lifting the problem from circular permutations to the affine symmetric group, the inversion distance can be found in polynomial time for a model in which inversions are restricted to acting on two regions.  This requires the proof of new results in group theory, and suggests a vein of new 
combinatorial problems concerning permutation groups
on which group theorists will be needed to collaborate with biologists.   We apply the new method to inferring distances and phylogenies for published \emph{Yersinia pestis} data.

\end{abstract}

\keywords{bacterial genomics, group theory,  hyperoctahedral group,  inversion distance,  phylogeny}

\subclass{20B35,  20B40,  92D15,  68R05}

\tableofcontents

\section{Introduction}

Significant variation in the position of genes in bacterial genomes is observed across individuals even within species~\cite{kawai2006genomes}. This variation is largely due to the process of inversion, which involves successive reversals of segments of the circular chromosome~\cite{eisen2000evidence}.
As the inversion process cannot be directly observed, mathematical models are needed to draw inferences from genomic data about these evolutionary processes. 
Such models can also be used to put a metric on the space of genome arrangements, that is, to establish estimates of evolutionary distance between genomes.
These distances in turn can be used to reconstruct phylogenetic trees using distance-based methods such as Neighbour-Joining~\cite{Saitou1987neighbour}. In the context of bacterial phylogenetics, inversions are particularly valuable for this purpose because, unlike comparisons via single nucleotide polymorphisms (SNPs), their signal is not affected by horizontal transfer events~\cite{Darling2008}.

Several algorithms have been successfully developed to obtain the minimal number of inversions (or other similar operations) required to transform one genome into another.  For instance, the breakpoint graph of~\citet{Bafna1993genome} and its successors (for example~\citet{hannenhalli1995transforming}) address the inversion distance problem very effectively under a model in which all inversions are equally probable.  There have even been first attempts to put this approach into an algebraic context~\cite{meidanis2000alternative}.
An alternative approach has been to define a wider set of operations than just inversions, which has led to an elegant general method called Double-Cut-and-Join (DCJ) for identifying the minimal distance between a pair of genomes~\cite{bergeron2006unifying}, again under the model in which all events are equally probable.

Because the processes involved in large-scale evolution on the bacterial chromosome are complicated and only partially understood, any algorithm to establish distance will have to make a range of simplifying assumptions about the processes in order to make any progress.  For instance, the basic approach of using the minimal number of inversions to estimate distance already implies that the rate of evolution is slow or that short evolutionary times are involved (which is justified in some cases~\cite{rocha2006inference}). Existing methods make further convenient assumptions about the inversion process itself, including the powerful assumption already mentioned that all inversions on a chromosome are equally probable.  An immediate consequence of this assumption is that inversions on a circular chromosome can be studied just as well on a linear chromosome, because any inversion on the circle is equivalent to another (equally probable) inversion of the complementary region, and hence we may ``cut" our representation of the circular permutation at some point and develop algorithms as though it were a linear chromosome.

Evidence has emerged recently that inversions are not all equally probable, and in particular that there are some restrictions on the length of an inverted region and on its location on the chromosome.  Specifically, on the one hand some studies have suggested that shorter inverted segments evolve more frequently~\cite{Dalevi2002measuring,Lefebvre2003detection,Darling2008}, and on the other it has been observed that the terminus of replication is always close to the antipode of the origin of replication~\cite{Darling2008,eisen2000evidence}.  The latter observation suggests that there is a fitness cost to having the terminus far from the antipode, and hence that individual inversions themselves do not move the terminus far from this position.

Such additional information is difficult to incorporate into existing solutions to the inversion distance problem, and suggests that a new approach is needed.  
In particular we give a simple example (Section~\ref{s:example}) that shows cutting the circle to linearize the problem will fail to find the minimal number of inversions when the length of an inverted region is constrained.
More generally, in this paper we show how the inversion process may be modelled as a group acting on the set of all possible genome arrangements. 
Groups are powerful ways of encapsulating the symmetries of objects.  While they have been intensively studied over the last century within pure mathematics, they have found many applications in science, especially in theoretical physics and domains such as crystallography where symmetry plays an obvious role.

We first present a group-theoretic framework in which some of the information about the inversion process can be realised in the form of models based on groups (Section~\ref{s:group.inversions}).  In this way we are able to translate questions about inversion distance into questions about groups.   
In the main body of the paper (Section~\ref{s:two.reg.model}) we address in detail one model of the inversion process and develop a new method to establish the inversion distance in this case, via the proof of results in group theory. In this model we consider the circular inversions limited to two regions. The key theoretical idea in our approach is to lift the problem from the group of circular permutations to the \emph{affine symmetric group}, which can be viewed intuitively as unrolling the circle on the infinite number line.  This is described in Sections~\ref{sec:affine} and~\ref{sec:background.affine}.  The main theoretical results are given in Sections~\ref{sec:crossings} and~\ref{sec:min.length.rep}.  In particular, Theorem~\ref{thm:minimal.is.shortest} shows how a same-length representative of a circular permutation in the affine symmetric group may be found.  Length can easily be calculated in the affine symmetric group, and so this resolves the problem with cutting the circular genome, and provides a framework in which other such models can be studied. 
As an example, we apply this method to the reconstruction of the phylogenetic history of some published \emph{Yersinia} genomes (Section~\ref{s:applications}).

\section{Group theoretic inversion systems}\label{s:group.inversions}

In this section we introduce a general group-theoretic framework that covers a range of models of the chromosomal inversion process.

\subsection{Group-theoretic inversion systems}\label{s:systems}

The idea of considering genomes as permutations of a set of regions is not new, and underlies all combinatorial methods.  Several other studies also put at least this much in the language of group theory~\cite{meidanis2000alternative,fertin2009combinatorics,moulton2011butterfly}.

For instance, given a pair of genomes for which eight conserved regions have been identified, the regions can be numbered according to the arrangement of one of the genomes so that this ``reference" genome is represented by the sequence $[1,2,3,4,5,6,7,8]$.  An inversion of the segment between regions 4 and 7 (inclusive) is then either $[1,2,3,7,6,5,4,8]$ (unsigned), or $[1,2,3,-7,-6,-5,-4,8]$ (signed).  This notation is based on a common two-line notation for permutations of $\n=\{1,2,\dots,n\}$, in which the top line gives the elements of the domain $\n$ and the bottom line gives the images of each element of $\n$, so that 
\[[1,2,3,-7,-6,-5,-4,8]\text{ is shorthand for  } \begin{pmatrix}1&2&3&4&5&6&7&8\\  1&2&3&-7&-6&-5&-4&8\end{pmatrix}.\]

The set of all such permutations forms a {group}, 
with each genome considered to be a group element (the reference genome being the identity element of the group).  
The group of unsigned permutations of a set $\n$ is a subgroup of the group of signed permutations (the \emph{hyperoctahedral} group, or Coxeter group of type $B$).

In general, we define an \emph{inversion system} to be a tuple $(G,\inv)$ where $G$ is a permutation group and $\inv$ is a set of inversions such that $\langle \inv \rangle=G$, i.e. $\inv$ generates $G$. 
If we choose a particular genome arrangement to be the reference genome, each possible genome can be considered to be a group element with respect to this reference.  
For any pair of genome arrangements there is a unique group element that transforms one into the other.  For instance if genomes $G_1$ and $G_2$ are represented by group elements (permutations $g_1$ and $g_2$), then the group element transforming $G_1$ into $G_2$ is $g_1^{-1}g_2$ (we write our groups acting on the right).  This group element is \emph{independent} of the choice of reference genome. 
In any inversion system \emph{the inversion distance problem is equivalent to finding the minimum length of a group element as a word in the generators}.
Given the impossibility of searching the entire group when the number of regions is large (for example, genomes with 60 regions can correspond to groups of order $\sim 10^{100}$), it is necessary to exploit the algebraic structure.

If all inversions are allowed and we ignore orientation, then
we have the group of all permutations of $\n$, namely the symmetric group $S_n$, the generators are all possible inversions of intervals on the circle, and the metric is the word length, up to the action of the dihedral group (see Section~\ref{s:circ.perms} for more details of this action).  This is the model considered by~\citet{Watterson-chrom-reversal-1982}, and for which they obtained bounds on the distance.  A significant number of extensions and improvements have followed this path with great success~\cite{kececioglu1993exact,hannenhalli1995transforming,bader2001linear}.  The signed version of this model, in which regions are regarded as having orientation, gives rise to the hyperoctahedral group, and is the most widely studied model in the inversion distance literature.

If all inversions are equally probable and are on the circle, an inversion of one region is equivalent to the inversion of the \emph{complementary} region.  One consequence is that one may consider a select region to be fixed, and only consider inversions that do not move it.  This enables treatment of the problem as if it were a linear chromosome.  This is the basis for many efficient methods currently available, including the use of breakpoints~\cite{kececioglu1993exact} and methods using the breakpoint graph due to~\citet{Bafna1993genome}.

In this paper we will study a non-uniform model, in which we only permit inversions of two adjacent regions.  In any such model in which the length of an inversion is restricted, it is not always possible to find the minimum number of inversions by cutting the genome and treating it as if it were linear.  This is best seen with an example.

\subsection{Circular permutations with restricted inversion size: an illustrative example}\label{s:example}

Consider the circular genome with eight regions numbered $1, \dots, 8$, and arranged around the circle in the order $[1,6,3,8,5,2,7,4]$.
This amounts to an arrangement in which each even numbered region has moved forwards two spaces and each odd backwards two spaces.  This arrangement can be sorted by 8 inversions of adjacent regions.  For instance if we write $s_i$ for the swap $(i\ i+1)$, with regions numbered mod 8 (so that $s_8$ means swapping the adjacent regions 8 and 1), this arrangement can be returned to the identity via the action of $s_8s_6s_4s_2s_7s_5s_3s_1$:
\begin{align*}
[1,6,3,8,5,2,7,4]s_8s_6s_4s_2s_7s_5s_3s_1
&=[4,6,3,8,5,2,7,1]s_6s_4s_2s_7s_5s_3s_1\\
&=[4,6,3,8,5,7,2,1]s_4s_2s_7s_5s_3s_1\\
&=[4,6,3,5,8,7,2,1]s_2s_7s_5s_3s_1\\
&=[4,3,6,5,8,7,2,1]s_7s_5s_3s_1\\
&=[4,3,6,5,8,7,1,2]s_5s_3s_1\\
&=[4,3,6,5,7,8,1,2]s_3s_1\\
&=[4,3,5,6,7,8,1,2]s_1\\
&=[3,4,5,6,7,8,1,2]
\end{align*}
which is the identity arrangement (on the circle), as shown in Figure \ref{fig:motivatingexample}.  The minimum number of inversions required to sort this arrangement by cutting-linearizing is 9. 
\begin{figure}[ht!]
\begin{center}
\includegraphics{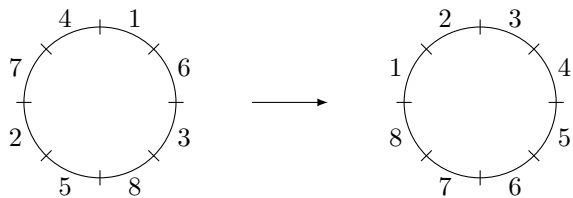}
\caption{An arrangement of eight regions whose length with 2-region inversions on the circle is 8 but when calculated by cutting and linearizing is 9.}
\label{fig:motivatingexample}
\end{center}
\end{figure}

\subsection{Circular permutations}\label{s:circ.perms}

The example above raises two important subtleties regarding circular permutations that do not affect linear rearrangements.  The first is that 
for regions on a circle numbered 1 to $n$, position  $n$ is adjacent to both position $n-1$ and position 1.   So, while an arrangement on the line with $n$ and $1$ swapped (that is $[n,2,3,\dots,n-1,1]$) requires many short inversions to sort, on a circle it requires just one 2-region inversion swapping $n$ and $1$.  This feature does not affect matters if all inversions are equally probable, because an inversion of the regions $2,\dots,n-1$ is just as likely as one of $n$ and 1, so the inversions across the boundary between $n$ and 1 can be ignored.  Of course, an inversion of the regions $2,\dots,n-1$ results in the regions apparently reversed in order: $[n,n-1,\dots,2,1]$.  However since our genome is in three dimensions this is equivalent to the arrangement $[1,2,\dots,n-1,n]$.  This brings to the fore the second subtlety.

Because the circular chromosome lives in three dimensions, we consider arrangements to be equivalent if they can be obtained from each other by rotating the circle or by turning it over (ignoring topological features such as knotting). This is because these actions on the circle simply correspond to different viewpoints from which to observe the chromosome.  The actions of rotating and reflecting a circle form a group (because they can be composed associatively and have inverse operations), and when we identify $n$ regularly positioned points on the circle the group is called the \emph{dihedral group}, denoted $\D_n$. The set of arrangements equivalent to $\s$ under the action of the dihedral group is the coset $\s\D_n$.  

On the one hand this means our search space --- the set of genuinely distinct arrangements --- is reduced by a factor of $2n$ (there are $2n$ elements of $\D_n$).  On the other hand, this means that to find the minimum number of inversions required to generate a given arrangement we need to consider each element of the coset, or in other words, each \emph{frame of reference}.  In what follows, our strategy is to develop a method to find the length of a permutation with respect to any given frame of reference, and then search over the $2n$ frames of reference for the shortest.

\begin{figure}[ht!]
\begin{center}
	\includegraphics{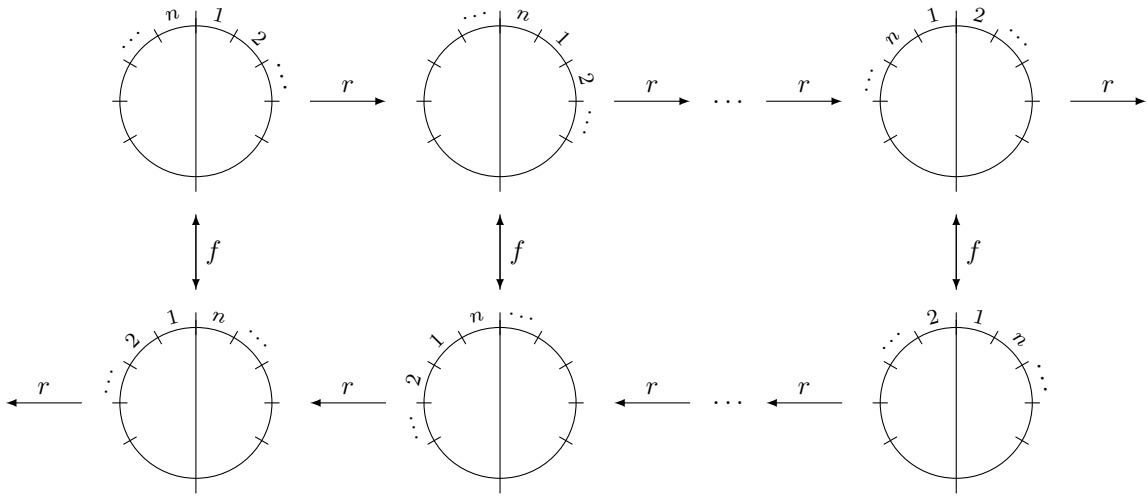}
\end{center}
\caption{Some of the $2n$ ``frames of reference" obtained by action of the dihedral group $D_n$ on the identity arrangement. Clockwise rotation is denoted by $r$ and the reflection in the vertical axis (flip) by $f$.}\label{fig:frames}
\end{figure}

We note that this subtlety is often not considered in the inversion distance literature (the initial paper by~\citet{Watterson-chrom-reversal-1982} is an exception).  This may be because in a model in which inversions of any length are equally probable, the circular chromosome can be treated as if it were linear: as remarked above, an inversion across the origin of replication is equivalent to one on the complementary region.  This feature is not enough, however, to account for the rotation and reflection symmetry, and to find the minimal distance under existing models it is still necessary to evaluate the distance under the different frames of reference (as shown in Figure~\ref{fig:frames}) obtained by the action of the $2n$ elements of the dihedral group.

\section{The two-region inversion distance problem for circular genomes}\label{s:two.reg.model}

For the remainder of the paper we focus on a constrained model in which only inversions of two adjacent regions are permitted, and in which we ignore orientation of the regions.  The restriction to two-region inversions allows us to exploit the theory of Coxeter groups.  In this model the group $G$ is the symmetric group $S_n$ and the generating set $\inv$ is the set of two-region inversions.  Note that because we are on the circle this is not a standard generating set for the symmetric group.  Notation and a group presentation of $(G,\inv)$ are given in Section~\ref{sec:2reg.model}.  

The key idea we exploit is that the best way to view circular permutations is by lifting the problem to the (extended) affine symmetric group.  This is the group of periodic permutations of the integers, and is a natural place to study circular permutations because by virtue of being on a circle, these are also periodic.  The affine symmetric group and the extended affine symmetric group are introduced in Sections~\ref{sec:affine} and~\ref{sec:background.affine}.

An important result that we use is a theorem of Shi (Theorem~\ref{t:shi.length}) that gives a formula for the length of an element of the affine symmetric group with respect to the standard generators.  Unfortunately, it cannot be used directly because despite the connection between circular and affine permutation groups, they are not identical --- for one thing the affine symmetric group is infinite and the circular permutation group is finite.  In fact the circular permutation group is a quotient of the affine symmetric group, and when we ``lift" a circular permutation to the affine situation (effectively looking at its pre-image) we have infinitely many choices.  The goal of 
Sections~\ref{sec:crossings} and~\ref{sec:min.length.rep} is to present a way to choose a representative of a circular permutation in the affine symmetric group that has the \emph{same length} as the circular permutation.  In this way we derive a method for finding the inversion-length of the circular permutation: choose a pre-image in the affine symmetric group that has the same length, and then apply Shi's formula.

The group-theoretic results in Sections~\ref{sec:crossings} and~\ref{sec:min.length.rep} are, as far as we are aware, new to algebra.  These are examples of answers to group-theoretic questions that arise directly from the biological models.  In other words, the group-theoretic questions asked, relating to lifts of circular permutations, are not ones that have arisen in other applications and so it is necessary to prove new claims.

In particular, in Section~\ref{sec:crossings} we prove that \emph{if} a lift of a circular permutation is of minimal length out of all possible lifts, then it satisfies some constraints on the ``crossings" (Definition~\ref{d:crossing.number}) that are possible.  These constraints are summarized in Corollaries~\ref{c.crossing.numbers} and~\ref{c.upper.lower.limits}.  Finally in Section~\ref{sec:min.length.rep} we prove the main result of this paper, Theorem~\ref{thm:minimal.is.shortest}, that in an affine permutation that is a minimal lift of a circular permutation, each element of $\n=\{1,2,\dots,n\}$ is moved a minimal amount.  This result (almost) determines the lift from the circular permutation to the affine symmetric group that gives the minimal length, as required.

\subsection{The two-region inversion model}\label{sec:2reg.model}

The group generators for this model are inversions of two regions, that is, simple transpositions $s_i=(i\ i+1)$ swapping regions $i$ and $i+1$.  However because we are on the circle with $n$ regions we consider regions to be equivalence classes of integers mod $n$, so that $s_n=(n\ n+1)$ is the transposition swapping regions $n$ and 1, and our group is the set of bijections on $\Z_n$.  

If we were to consider only the generators $s_1,\dots s_{n-1}$ then this group would be precisely the standard presentation for the symmetric group as a Coxeter group of type $A$.  The relations in this case are 
\begin{align*}
s_i^2&=1&&\text{for each $i=1,\dots,n-1$}, \\
s_is_j&=s_js_i&&\text{if }|i-j|>1,\text{ and}\\
s_is_{i+1}s_i&=s_{i+1}s_is_{i+1}&&\text{for each }i=1,\dots,n-2.
\end{align*}

The additional generator $s_n$ can be written in terms of these generators as the conjugate 
\[s_n=s_{n-1}s_{n-2}\dots s_2s_1s_2\dots s_{n-2}s_{n-1}\]
and it behaves in a precisely analogous manner to the above, giving a presentation (via Tietze transformations) for our circular permutation group with generators $\{s_i\mid i=1,\dots,n\}$ and relations:
\begin{align*}
s_i^2&=1&&\text{for each $i=1,\dots,n$}, \\
s_is_j&=s_js_i&&\text{if }i-j\mod n\neq\pm 1,\text{ and}\\
s_is_{i+1}s_i&=s_{i+1}s_is_{i+1}&&\text{for each }i=1,\dots,n-1.
\end{align*}

While the length function for the symmetric group with respect to the Coxeter-type presentation is well-studied (see for example~\citet{Hum90}), the additional generator provided by the circle makes the relevant theorems inapplicable.  Fortunately, there is another way to view this group using the affine symmetric group $\tilde S_n$.

In the remainder of this section we show how to lift a circular permutation to the affine symmetric group in such a way that the length in the affine symmetric group is the minimal number of 2-region inversions required to sort the circular permutation (see Figure \ref{fig:liftingstrategy}).   Because we know how to calculate length in the affine symmetric group (see next section), we can then use this to find the length of the circular permutation.
We will in fact be working often in the \emph{extended} affine symmetric group $\tilde S_n^{ext}$, which we introduce in Section~\ref{sec:background.affine}.  This group is defined in the same way as the affine symmetric group, but without the restriction on the sum of the images of a permutation given by Equation~\eqref{eq:balance} below.

\subsection{The affine symmetric group}\label{sec:affine}

\begin{figure}
\begin{center}
\includegraphics{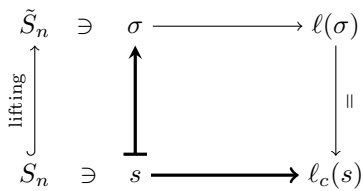}
\end{center}
\caption{The strategy of lifting a circular permutation $s$ to an affine permutation $\sigma$ in order to calculate its length. Thick arrows denote the mappings that are constructed in this paper.}
\label{fig:liftingstrategy}
\end{figure}
Recall that $\n:=\{1,2,\dots,n\}$.
\begin{definition}[affine permutation]\label{def:affine.perm}
A bijection $\sigma:\Z\to\Z$ is an affine permutation if it satisfies the following two conditions:
\begin{description}
\item[periodic:]\begin{equation}\sigma(i+n)=\sigma(i)+n\text{ for all }i\in\Z\end{equation} 
\item[balanced:] 
\begin{equation}\label{eq:balance}
   \sum_{i\in\n } \sigma(i)
     \;=\; \sum_{i\in\n } i
     \;=\; \frac{n(n+1)}2
\end{equation}
\end{description}
following \citet{lusztig1983squareintegrable}.
\end{definition} 
The affine symmetric group $\tilde S_n$ can be realized as the group of affine permutations.  The symmetric group $S_n$ can be identified with a subgroup of $\tilde S_n$ by extending each permutation from $\n $ to $\Z$ obeying the periodicity requirement, and can also be obtained by projecting from $\tilde S_n$ by reducing domain and codomain mod $n$.  
In particular the generators $\tilde s_i$ of $\tilde S_n$ are the periodic permutations $\tilde s_i(i+kn)=i+1+kn$ and $\tilde s_i(i+1+kn)=i+kn$ for all $k\in\Z$, and $\tilde s_i(j)=j$ if $j\not\equiv i$ or $i+1\mod n$.

As with the symmetric group in relation to the Coxeter presentation, there is an easy formula for the minimum number of (affine) transpositions $\tilde s_i$ required to represent an affine permutation, due to~\citet{shi1986kazhdan} (see also \citet{shi1994presentations}).  In our context, using the length function in $S_n$ would not take into account the additional generator $s_n$; the key idea here is that by lifting a circular permutation up to the affine symmetric group the length function in $\tilde S_n$ accounts for all $n$ generators.

The obstacle to using the length function in the affine symmetric group is that a circular permutation doesn't define a unique affine permutation: we need to choose for each $s(i)\in\Zn$ a representative $\s(i)\in\Z$, and we need to do this in such a way that the length of the obtained affine permutation is minimised.

To illustrate the problem, consider the circular permutation 
\[\s=\begin{pmatrix}1&2&3&4&5\\ 3&5&4&1&2\end{pmatrix},\] 
which we denote in ``window" notation by $[3,5,4,1,2]$.  This has $3\mapsto 4$, but lifting to $\tilde S_5$ we could have any choice of \(3\mapsto\{4+5j\mid j\in\Z\}\) (see Figure~\ref{fig:lift.to.affine}).  All choices give equivalent circular permutations.

\begin{figure}[ht!]
\includegraphics[width=\textwidth]{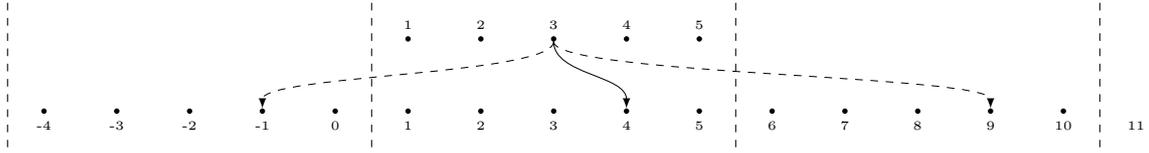}
\caption{Some of the infinitely many alternative ways to lift the mapping $3\mapsto 4$ to the affine symmetric group.}
\label{fig:lift.to.affine}
\end{figure}

\subsection{The extended affine symmetric group}\label{sec:background.affine}

Some choices of how to lift a circular permutation to an affine permutation will result in periodic permutations that do not satisfy condition (2) of Definition~\ref{def:affine.perm}: that is, they are not balanced.  Periodic permutations that are not balanced are part of the \emph{extended} affine symmetric group, as defined in this subsection.  We also give the Shi length formula (Theorem~\ref{t:shi.length}) that holds for both the extended and the non-extended affine symmetric groups, and define the crossing number of an affine permutation (Definition~\ref{d:crossing.number}).

\begin{definition}[extended affine permutation]\label{d:balanced}
An extended affine permutation is a periodic bijection $\sigma:\Z\to\Z$.
\end{definition}
That is, a \emph{balanced} extended affine permutation is an affine permutation: an element of $\tilde S_n$.  The set of extended affine permutations also forms a group, denoted $\tilde S_n^{ext}$, and it can be generated by the same set $\tilde s_i$ for $i\in\n$ together with the additional permutation $\tau:\Z\to\Z$ given by $\tau(i)=i+1$ for all $i\in\Z$.  

Any extended affine permutation $\s$ can be balanced by multiplication with a suitable power of~$\tau$.  Importantly, this power of $\tau$ does not affect the number of transpositions $\tilde s_i$ needed to express $\s$:  any extended affine permutation may be written as a product $\tau^kw$ where $w$ is an expression in the $\tilde s_i$ for $i=1,\dots,n$, and $k$ as well as the length of $w$ are unique.  The unique power of $\tau$ in such an expression is called the \emph{winding number} of the permutation.

In the context of circular permutations, $\tau$ corresponds to a rotation of the circle by one position.

Let $\PP = \big\{ (i,j)\in\n ^2 \mid i<j \big\}$, i.e.\ the set of all those pairs of regions that are in strictly ascending order.

\begin{theorem}[\citet{shi1986kazhdan}]\label{t:shi.length}
Given an extended affine permutation $\s$, the length $\ell(\s)$ of $\s$, that is the minimum number of transpositions occurring in any representation of $\s$ as a product of transpositions, is given by

\begin{equation}
\label{eq:shilength}
  \ell(\s)
    \;=\; \sum_{(i,j)\in\PP} \left|\left\lfloor \frac{\sigma(j)-\sigma(i)}{n}\right\rfloor\right|
  \;.
\end{equation}
\end{theorem}
\begin{proof}
See Lemma 4.2.2 of \citet{shi1986kazhdan}, pages 68--70.
\end{proof}

\begin{definition}\label{d:crossing.number}
Given an extended affine permutation $\s$, and $(i,j)\in\PP$, we define the \textbf{crossing number} $\kappa_{i,j}(\s)$ of the positions $i$ and $j$ in $\s$ as
\[
  \kappa_{i,j}(\s)
    \;=\; \left\lfloor \frac{\s(j)-\s(i)}{n} \right\rfloor
  \;.
\]
Associating to $\s$ a diagram as in Figure~\ref{fig:lift.to.affine}, the crossing number for the pair $(i,j)$, $\kappa_{i,j}(\s)$, is the number of strands $(j+kn)\mapsto\s(j+kn)$ ($k\in\Z$) that cross the strand $i\mapsto\s(i)$, where crossings from the left are counted positive and crossings from the right are counted negative.
\end{definition}

\begin{definition}
For $k\in\Z$ we define the set of all strictly ascending pairs with crossing number exactly $k$ in $\s$ as
\[
  I_k(\s)
    \;=\; \big\{ (i,j)\in \PP \mid \kappa_{i,j}(\s)=k \big\}
    \;=\; \big\{ (i,j)\in \PP \mid kn\le \s(j)-\s(i)<(k+1)n \big\}
  \;.
\]
\end{definition}

Note that 
\begin{align}
    \sum_{(i,j)\in\PP} \left|\left\lfloor \frac{\sigma(j)-\sigma(i)}{n}\right\rfloor\right|
    \;&=\; \sum_{(i,j)\in\PP} \left| \kappa_{i,j}(\s) \right|\notag\\
    \;&=\; \sum_{k\in\Z} |k|\cdot|I_k(\s)|\notag\\
    \;&=\; \sum_{k\in\N^+} |k|\cdot\big(|I_k(\s)|+|I_{-k}(\s)|\big).\label{eq:affine.length.var}
\end{align}

\subsection{Crossing numbers for minimum length representatives}\label{sec:crossings}

Thanks to Shi's formula (Theorem~\ref{t:shi.length}), the crossing numbers are closely related to length.  So given a minimum length representative of a circular permutation, there should be some constraints on the crossing numbers.  
In this section we show that the only crossing numbers that can occur in a minimum length representative of a given circular permutation are $-1$, $0$, and $+1$.  

We start by proving some ``transitivity'' constraints for pairs of crossing numbers involving a common position.  For instance we provide constraints on the crossing number of the pair $(i,k)$ when we have crossing numbers for $(i,j)$ and $(j,k)$ (Part (i)).

\begin{lemma}\label{l:crossing.number.transitivity}
Let $\s$ be an extended affine permutation.  
The following hold:
\begin{enumerate}[\upshape{(}i\upshape{)}]
\item
 If $(i,j)\in I_r(\s)$ and $(j,k)\in I_s(\s)$ then $(i,k)\in I_{r+s}(\s)\cup I_{r+s+1}(\s)$.

\item
If $(i,j)\in I_r(\s)$ and $(i,k)\in I_s(\s)$ then $(j,k)\in I_{s-r-1}(\s)\cup I_{s-r}(\s)$ or $(k,j)\in I_{r-s-1}(\s)\cup I_{r-s}(\s)$.

\item
If $(i,k)\in I_r(\s)$ and $(j,k)\in I_s(\s)$ then $(i,j)\in I_{r-s-1}(\s)\cup I_{r-s}(\s)$ or $(j,i)\in I_{s-r-1}(\s)\cup I_{s-r}(\s)$.

\end{enumerate}
\end{lemma}
\begin{proof}
\begin{enumerate}[\upshape{(}i\upshape{)}]
\item
 As $(i,j)\in I_r(\s)$ is equivalent to $rn\le \s(j)-\s(i)<(r+1)n$ and
 $(j,k)\in I_s(\s)$ is equivalent to $sn\le \s(k)-\s(j)<(s+1)n$, one has
 $(r+s)n\le \s(k)-\s(i)<(rs+2)n$, whence $(i,k)\in I_{r+s}(\s)\cup I_{r+s+1}(\s)$.
\item
 As $(i,j)\in I_r(\s)$ is equivalent to $-(r+1)n<\s(i)-\s(j)\le -rn$ and
 $(i,k)\in I_s(\s)$ is equivalent to $sn\le \s(k)-\s(i)<(s+1)n$, one has
 $(s-r-1)n<\s(k)-\s(j)<(s-r+1)n$, respectively $(r-s-1)n<\s(j)-\s(k)<(r-s+1)n$.
 If $j<k$, the former implies $(j,k)\in I_{s-r-1}(\s)\cup I_{s-r}(\s)$,
 while in the case $k<j$, the latter implies $(k,j)\in I_{r-s-1}(\s)\cup I_{r-s}(\s)$.
\item
 As $(i,k)\in I_r(\s)$ is equivalent to $rn\le \s(k)-\s(i)<(r+1)n$ and
 $(j,k)\in I_s(\s)$ is equivalent to $-(s+1)n<\s(j)-\s(k)\le -sn$, one has
 $(r-s-1)n<\s(j)-\s(i)<(r-s+1)n$, respectively $(s-r-1)n<\s(i)-\s(j)<(s-r+1)n$.
 If $i<j$, the former implies $(i,j)\in I_{r-s-1}(\s)\cup I_{r-s}(\s)$,
 while in the case $j<i$, the latter implies $(j,i)\in I_{s-r-1}(\s)\cup I_{s-r}(\s)$.
\end{enumerate}
\end{proof}

The idea now is the following:  Assume that $\alpha>1$ is the maximal integer such that $I_{-\alpha}(\s)\cup I_{\alpha}(\s)$ is non-empty.  This means that $|\alpha|$ is the maximal size of a crossing number in $\s$.  
We will define two transformations that reduce this, by removing the pairs with crossing numbers $-\alpha$ (respectively~$\alpha$) without increasing the absolute value of any crossing number, except from 0 to 1.  The proof of this claim uses the transitivity constraints in Lemma~\ref{l:crossing.number.transitivity} and the maximality of $\alpha$.

Given an extended affine permutation $\s$ and a subset $S\subseteq\n $, we define the extended affine permutation $\s^S$ by setting
\[
  \s^S(i) =
     \begin{cases}
        \s(i) + n & \text{if $i\in S$,} \\
        \s(i)     & \text{otherwise.}
     \end{cases}
  \;
\]
That is, $\s^S$ has the image of each element of $S$ shifted by $n$.  Note that this does not change the circular permutation that they correspond to.

\begin{lemma}\label{l.add.subtract.n}
For an extended affine permutation $\s$, if $\kappa_{i,j}(\s)=r$, then one has
\[
  \kappa_{i,j}(\s^S) =
     \begin{cases}
        r+1 & \text{if $i\notin S$ and $j\in S$,} \\
        r-1 & \text{if $i\in S$ and $j\notin S$,} \\
        r   & \text{otherwise.}
     \end{cases}
  \;
\]
\end{lemma}
\begin{proof}
The claims follow immediately from Definition~\ref{d:crossing.number}.
\end{proof}

The following two propositions show how to choose a subset $S\subseteq\n$ in such a way that the largest crossing number for $\s^S$ is strictly lower than that of $\s$.

\begin{proposition}\label{p.kill.max}
Let $\s$ be an extended affine permutation, and assume that $\alpha>1$ is the maximal integer such that $I_{\alpha}(\s)\cup I_{-\alpha}(\s)$ is non empty.

For $S=\big\{ i\in\n  \mid \exists\, j\in\n  \,,\; (i,j)\in I_{\alpha}(\s)\big\}$ one has $I_{\alpha}(\s^S)=\varnothing$.  Moreover, for any $(i,j)\in\PP$ with $\kappa_{i,j}(\s)\ne0$, one has $|\kappa_{i,j}(\s^S)| \le|\kappa_{i,j}(\s)|$.
\end{proposition}

\begin{proof}
Assume that $(i,j)\in I_r(\s)$, so that $\kappa_{i,j}(\s)=r$, and suppose 
\[|\kappa_{i,j}(\s^S)| > |\kappa_{i,j}(\s)| = |r| > 0.\]
Then, by Lemma~\ref{l.add.subtract.n}, if $r>0$ we have $i\notin S$ and $j\in S$, and if  $r<0$ we have $i\in S$ and $j\notin S$.

If $r>0$ then there exists $k\in\{j+1,\ldots,n\}$ with $(j,k)\in S=I_{\alpha}(\s)$, so $(i,k)\in I_{\alpha+r}(\s)\cup I_{\alpha+r+1}(\s)$ by Lemma~\ref{l:crossing.number.transitivity} contradicting the maximality of $\alpha$.

If $r<0$ then there exists $k\in\{i+1,\ldots,n\}$ such that $(i,k)\in S=I_{\alpha}(\s)$.
By Lemma~\ref{l:crossing.number.transitivity}, either $j<k$ and $(j,k)\in I_{\alpha-r-1}(\s)\cup I_{\alpha-r}(\s)$,
or $k<j$ and $(k,j)\in I_{r-\alpha-1}(\s)\cup I_{r-\alpha}(\s)$. Since $r<0$, the only situation compatible with the maximality of $\alpha$ is $r=-1$ and $(j,k)\in I_{\alpha}(\s)$, but in this case $j\in S$, which is a contradiction.

By Lemma~\ref{l.add.subtract.n}, it only remains to show that $\kappa_{i,j}(\s^S)\ne\kappa_{i,j}(\s)$ for any $(i,j)\in I_{\alpha}(\s)$.  As $(i,j)\in I_{\alpha}(\s)$ implies $i\in S$, one can only have $\kappa_{i,j}(\s^S)=\kappa_{i,j}(\s)$ if $j\in S$, that is, if there  exists $k\in\{j+1,\ldots,n\}$ with $(j,k)\in S=I_{\alpha}(\s)$.  But then, by Lemma~\ref{l:crossing.number.transitivity}, $(i,k)\in I_{2\alpha}(\s)\cup I_{2\alpha+1}(\s)$, again contradicting the maximality of $\alpha$.
\end{proof}

A similar argument deals with the case in which $I_{\alpha}(\s)$ is empty:

\begin{proposition}\label{p.kill.min} 
Let $\s$ be an extended affine permutation, and assume that $\alpha>1$ is the maximal integer such that $I_{\alpha}(\s)\cup I_{-\alpha}(\s)$ is non empty.  Moreover assume that $I_{\alpha}(\s)=\varnothing$.

For $S=\big\{ j\in\n  \mid \exists\, i\in\n  \,,\; (i,j)\in I_{-\alpha}(\s)\big\}$ one has $I_{-\alpha}(\s^S)=\varnothing$.  Moreover, for any $(i,j)\in\PP$ with $\kappa_{i,j}(\s)\ne0$, one has $|\kappa_{i,j}(\s^S)| \le|\kappa_{i,j}(\s)|$.
\end{proposition}
\begin{proof}
The proof is similar to that of Proposition~\ref{p.kill.max}.
\end{proof}

We now translate these results into the context of a circular permutation $s$.  In particular, we have the key conclusion that if $\s$ is a minimal affine representative of $s$ then its crossing numbers are at most $\pm 1$.

\begin{corollary}\label{c.crossing.numbers}
If $s$ is a circular permutation and $\s$ is an extended affine permutation representing $s$, whose length is minimal among all representatives of $s$, then $I_k(\s)=\varnothing$ for all $k\in\Z\setminus\{-1,0,+1\}$.
\end{corollary}
\begin{proof}
If $\sigma$ is a representative of $s$ and $I_k(\s)\ne\varnothing$ for some $k\in\Z\setminus\{-1,0,+1\}$, then application of Proposition~\ref{p.kill.max} or Proposition~\ref{p.kill.min} produces another representative of $s$ that has, by Theorem~\ref{t:shi.length}, smaller length.
\end{proof}

\begin{corollary}\label{c.upper.lower.limits}
If $s$ is a circular permutation and $\s$ is an extended affine permutation representing $s$, whose length is minimal among all representatives of $s$, then $\max\big\{\s(i)\mid i\in\n \big\}-\min\big\{\s(i)\mid i\in\n \big\} < 2n$.
\end{corollary}
\begin{proof}
Choose $k\in\n $ such that $\s(k)=\min\big\{\s(i)\mid i\in\n \big\}$.
One has $\s(j)-\min\big\{\s(i)\mid i\in\n \big\}=\s(j)-\s(k)<2n$ for all $j\in\n $ by Corollary~\ref{c.crossing.numbers}, so the claim holds.
\end{proof}

We have now placed significant constraints on an important feature of an affine permutation (its crossing numbers), when it is a minimal length representative of a circular permutation.  The question remains of how to choose an affine representative that satisfies these constraints on the crossing numbers.  This is addressed in the next section.

\subsection{Finding a minimum length representative}\label{sec:min.length.rep}

In this section we show that a minimal affine representative of a circular permutation must have images for each $i\in\n$ that are the minimal possible distance from $i$ (Theorem~\ref{thm:minimal.is.shortest}).  To begin with, we will need to define the \emph{nett} crossing number of a position $i\in\n$ (Definition~\ref{d:nett.crossing.num}), as well as the \emph{winding number} of $\s$.

\begin{definition}\label{d:nett.crossing.num}
Given an extended affine permutation $\s$, and $i\in\n $, we define the \textbf{nett crossing number} $\nu_{i}(\s)$ of the position $i$ to be
\[
  \nu_{i}(\s) \;=\; \sum_{j\in\n } \kappa_{i,j}(\s)
  \;.
\]
\end{definition}

Any vertical line in general position through a diagram in $\tilde S_n^{ext}$ has the same number of nett crossings (crossings from the left minus crossings from the right, or vice versa).  This number is the \emph{winding number} of the permutation, and is the exponent on $\tau$ in its expression in terms of the $s_i$ and $\tau$, as described in Section~\ref{sec:background.affine} (the elements of winding number zero are the affine permutations, that is, the ``balanced'' ones).  This feature follows because each generator $s_i$ has zero nett crossings, and nett crossings are invariant under Reidemeister moves.

We now have a lemma that is a direct consequence of the results in the previous section, showing that in minimal representatives the images of regions cannot move more than $n$.  This is somewhat intuitive, since if it were not true it would mean that in a minimal sequence of inversions a region could move more than a full circle to its destination, which (intuitively) seems unlikely. 

\begin{lemma}\label{lem:close.image}
If $\s$ is a minimal balanced representative of a circular permutation, then $|\s(i)-i|<n$.
\end{lemma}

\begin{proof}
If $|\s(i)-i|\ge n$ then there are two alternatives: $\s(i)\ge i+n$ or $\s(i)\le i-n$.  We show that either leads to a contradiction.

Suppose that $\s(i)\ge i+n$.  Then $\s(i-n)\ge i$ by periodicity.
If there was a $j>i$ that had an image $\s(j)$ less than $\s(i-n)=\s(i)-n$, then it would cross the $i$ strand twice, violating Lemma~\ref{c.crossing.numbers}.
This implies that the winding number of $\s$ is strictly positive, because in the vertical line down from $i+\varepsilon$ (for $\varepsilon$ sufficiently small), the only crossings can be from the left.  This violates the balance of $\s$ and so is a contradiction.

The second case is symmetric.
\end{proof}

The following is a technical consequence of Lemma~\ref{lem:close.image} that is needed in the next proof: 

\begin{lemma}\label{lem:add.jn}
Suppose $j\in\Z\setminus\{0\}$.  In a minimal balanced representative, 
\[
|\s(i)-i+jn|=\begin{cases}
\s(i)-i+jn&\ \text{if }j>0\\
-(\s(i)-i+jn)&\ \text{if }j<0.
\end{cases}
\]
\end{lemma}
\begin{proof}
This is immediate from Lemma~\ref{lem:close.image}, since either $0<\s(i)-i<n$ or $-n<\s(i)-i<0$.
\end{proof}

\begin{lemma}\label{lem:nett.crossings}
If $\s$ is a balanced permutation then $\nu_i(\s)=i-\s(i)$.  Moreover, if $\s$ has winding number $k$ then $\nu_i(\s)=i-\s(i)+k$.
\end{lemma}
\begin{proof}
If $\s$ is balanced, the element $\s\tau^{i-\s(i)}$ of $\tilde S_n^{ext}$ sends $i\mapsto i$, and has winding number $i-\s(i)$.  A vertical line in general position just to the left or right of $i$ will have the same number of nett crossings as the strand $i\mapsto i$, namely the winding number $i-\s(i)$.  But the multiplication of $\s$ by $\tau^{i-\s(i)}$ does not change any crossings, and hence the claim follows.  When the winding number of $\s$ is $k\neq 0$, the winding number after multiplication by $\tau^{i-\s(i)}$ is $k+(i-\s(i))$ as required. 
\end{proof}

\begin{theorem}\label{thm:minimal.is.shortest}
If $s$ is a circular permutation and $\s$ is an affine permutation representing $s$, whose length is minimal among all representatives of $s$, then for each $i\in\n $, $\s$ takes the shortest distance between $i$ and the equivalence class $\s(i)\mod n$.
\end{theorem}

\begin{proof} 
We have a circular permutation $s$ and a balanced, minimal length lift, $\s\in\tilde S_n$.  Our claim is that each $\s(i)$ is the minimal distance from $i$ of all choices $\{\s(i)+jn\mid j\in\Z\}$.  

For an arbitrary $i$, we consider alternative choices of image from the set $\{\s(i)+jn\mid j\in\Z\setminus\{0\}\}$.  We claim that any of these alternatives increases the distance of the image from $i$. 

Given that $\s$ is balanced, a choice of $i\mapsto\s(i)+jn$ results in winding number $j$, and so the nett number of crossings of the line $i\mapsto\s(i)+jn$ is $i-(\s(i)+jn)+j$, by Lemma~\ref{lem:nett.crossings}.  If $j>0$, all crossings of this edge are from the right, as crossings from the left would cross the edge $i\mapsto\s(i)$ twice, violating minimality by Corollary~\ref{c.crossing.numbers} (see Figure~\ref{fig:hypoth.crossings}).  
Similarly, if $j<0$ all crossings are from the left, for the same reason.

\begin{figure}[ht]
\begin{center}
\includegraphics{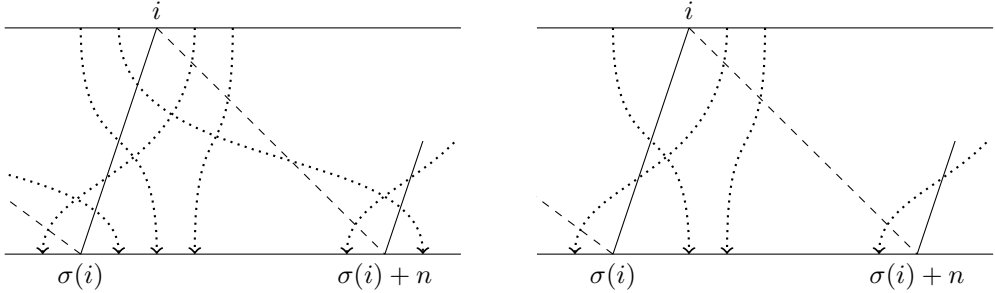}
\end{center}
\caption{\small The dotted arrows in the left panel indicate all possible sets of connections that cross the strand $i\mapsto\s(i)$ and a hypothetical alternative strand $i\mapsto\s(i)+n$ (shown dashed).  If $\s$ is minimal then no pair of equivalence classes of strands cross more than once, by Corollary~\ref{c.crossing.numbers}.  This means one of the sets of connections drawn on the left is empty and the only possible non-empty sets are shown on the right panel.  Note that the panel on the right shows crossings of the hypothetical alternative strand that are all from the same side (right to left).}
\label{fig:hypoth.crossings}
\end{figure}

It follows that the total number of crossings of the edge $i\mapsto\s(i)+jn$ equals the absolute value of the nett number of crossings, namely 
\begin{align*}
\text{total crossings of }(i\mapsto\s(i)+jn)
&=
\begin{cases}
-(i-(\s(i)+jn)+j)&\text{ if }j>0\\
i-(\s(i)+jn)+j&\text{ if }j<0.
\end{cases}\\
&=
\begin{cases}
(\s(i)-i)+jn-j&\text{ if }j>0\\
(i-\s(i))-jn+j&\text{ if }j<0.
\end{cases}
\end{align*}

For the original strand we have $i-\s(i)$ nett crossings and distance $|\s(i)-i|$.  So the total number of crossings of the strand is at least $|i-\s(i)|$. Therefore we have the inequalities:
\begin{align*}
|i-\s(i)|
&\le \text{total crossings of }(i\mapsto\s(i))\\
&\le \text{total crossings of }(i\mapsto\s(i)+jn)\text{ by minimality of $\s$,}\\
&=\begin{cases}
(\s(i)-i)+jn-j&\text{ if }j>0\\
(i-\s(i))-jn+j&\text{ if }j<0,
\end{cases}\\
&<\begin{cases}
(\s(i)-i)+jn&\text{ if }j>0\\
(i-\s(i))-jn&\text{ if }j<0,
\end{cases}\\
&=|\s(i)+jn-i|\quad\text{ by Lemma~\ref{lem:add.jn}}\\
&=\text{the distance from $i$ to $\s(i)+jn$}.
\end{align*}
In other words, if $\s$ is a balanced, minimum length representative of the circular permutation $s$, then each $\s(i)$ is the minimum distance from $i$ of all the alternatives $\{\s(i)+jn\mid j\in\Z\}$, as required.

\end{proof}

Note: the distance between $i$ and $\s(i)$ is strictly less than that between $i$ and $\s(i)+n$ when $\s$ is drawn with the winding number zero.  If $\s(i)+n$ is instead chosen as the image of $i$ the winding number becomes 1, as noted in the proof.  If the permutation is rebalanced (changing the frame of reference) then the bottom axis is moved one to the left, resulting in the distance $i$ to $\s(i)$ increasing by one and the distance $i$ to $\s(i)+n$ decreasing by 1.  Consequently, some choices $\s(i)$ or $\s(i)+n$ could be equivalent.  This occurs when $\s(i)=(n-1)/2$ ($n$ odd).  Here is the argument.

\begin{lemma}
If two choices $i\mapsto\s(i)$ and $i\mapsto\s(i)+n$ both result in minimal length elements, and the permutation is balanced with $i\mapsto\s(i)$, then $i-\s(i)=\frac{1}{2}(n-1)$.
\end{lemma}
\begin{proof}
As in the proof of Theorem~\ref{thm:minimal.is.shortest}, if $i\mapsto\s(i)$ gives a minimal length element then the number of crossings of a line $i\mapsto\s(i)+n$ is $\s(i)+n-i-1$.  Similarly, if $i\mapsto\s(i)+n$ also gives a minimal length element then the number of crossings of a line $i\mapsto\s(i)$ is $i-\s(i)$.  Since both choices have the same total number of crossings and all else remains fixed, these lines must have the same number of crossings, namely $\s(i)+n-i-1=i-\s(i)$.  The lemma follows.
\end{proof}

Note that in the above lemma, and as mentioned prior to it, the permutation is balanced with $i\mapsto\s(i)$ but not balanced with $i\mapsto\s(i)+n$.  So the distances when balanced are not the same thing as the number of crossings.  The above scenario arises when $n$ is odd, so for a given frame of reference there is only one choice: the distances $i$ to $\s(i)$ and $i$ to $\s(i)+n$ are different.

The results in this section show that, for each frame of reference, a minimal representative may be found by choosing shortest distances for each image.  Taking the shortest representative over all frames of reference will yield the minimal number of inversions required for the given circular permutation.

\section{Implementation and Application}\label{s:applications}

In this section we explain how the results may be implemented algorithmically to compute the inversion length in the two-region inversion model, and then apply the method to some published \emph{Yersinia pestis} genomes.  

\subsection{Computational implementation}

The method arising from these results breaks into three natural algorithmic components: 
\begin{enumerate}
\item minimizing paths for the lifting process, 
\item calculating the length for an affine permutation, and
\item sorting a circular permutation (finding an explicit sequence of inversions).
\end{enumerate}

 For lifting a circular permutation into an extended affine permutation we have to make a decision which way to route each path, i.e.\ choosing minimal-distance images for each $i\in \n$.  A straight-forward method checks for each $i\in\n$ whether the image in the previous or in the next window has a shorter distance or not (Alg.\ \ref{alg:pathminim}).

For calculating the length of an affine permutation we simply count the number of crossings. The sum in Equation \eqref{eq:shilength} can be translated into a {\texttt{for}} loop easily. Then for calculating the circular length we have to go through all frames of reference to find the minimal length (Alg.\ \ref{alg:length}).  

Additionally, we can \emph{sort} the permutation, producing a geodesic.
The sorting algorithm (Alg. \ref{alg:sort}) operates by comparing consecutive pairs and swaps them if needed, hence doing uncrossings. This algorithm always chooses the lowest index pair to be swapped and thus produces a single geodesic. However, a systematic exploration of all possible swaps (e.g.\ a backtrack algorithm) can enumerate geodesics.   

Since we combine quadratic and linear algorithms, the overall sorting algorithm is also polynomial. In particular we can easily deal with real-world genome data with approximately 80 regions. The algorithms in this subsection were implemented using the \textsf{GAP} \cite{gap4} computer algebra system and the source code is available upon request.

\begin{algorithm}[ht!]
\SetKwInOut{Input}{input}\SetKwInOut{Output}{output}
\Input{$n$ number of points\\
  $i$ a point\\
  $\sigma(i)$ the image of $i$ under $\sigma$}
\Output{$\sigma'(i)$ the minimized image}
\SetKwFunction{MinimizePath}{MinimizePath}
\SetKwInOut{Name}{\MinimizePath($n,i,\sigma(i)$)}
\BlankLine
\Name{}
$d\leftarrow  |i-\sigma(i)|$\;
\lIf{$|i-(\sigma(i)+n)|<d$}{return $\sigma(i)+n$\;}
\lIf{$|i-(\sigma(i)-n)|<d$}{return $\sigma(i)-n$\;}
return $\sigma(i)$\;
\caption{Minimizing the path between a point and its image.}
\label{alg:pathminim}
\end{algorithm}

\begin{algorithm}[ht!]
\SetKwInOut{Input}{input}\SetKwInOut{Output}{output}
\Input{$\s$ circular permutation}
\Output{the 2-inversion length of $\s$}
\SetKwFunction{CTIL}{CircularTwoInversionLength}
\SetKwInOut{Name}{\CTIL($\s$)}
\SetKwData{Min}{min}
\SetKwFunction{Size}{Size}
\BlankLine
\Name{}
$n\leftarrow$ \Size($\s$);
$\Min\leftarrow \ell(\s)$\;
\ForEach{$\s'\in \s\cdot D_n$}{
  $\s'_m\leftarrow$ all paths minimized in $\s'$\;
  \If{$\ell(\s'_m)<$\Min}{
    \Min$\leftarrow\ell(\s'_m)$\;
  }
}
return $\Min$\;
\caption{Calculating the 2-inversion length of a circular permutation.}
\label{alg:length}
\end{algorithm}

\begin{algorithm}[ht!]
\SetKwInOut{Input}{input}\SetKwInOut{Output}{output}
\Input{$m$ list of images of minimized configuration}
\Output{$w$ a word encoding the sequence of 2-inversions for sorting $m$
}
\SetKwData{Finished}{finished}
\SetKwFunction{Swap}{Swap}
\SetKwFunction{Size}{Size}
\SetKwFunction{Add}{Add}
\SetKwFunction{MinimizePath}{MinimizePath}
\SetKwFunction{Sort}{Sort}
\SetKwInOut{Name}{\Sort($m$)}
\BlankLine
\Name{}
$n\leftarrow$ \Size($m$);$w\leftarrow$ []\;
\ForEach{$i\in \n $}{$d[i]\leftarrow m[i]-i$\;}
\Repeat{\Finished}{
  \Finished$\leftarrow$ true\;
  \ForEach{$i\in\n$}{
    $j\leftarrow(i+1)\mod n$\;
    \If{$d[i]>d[j]$}{
      \Swap($m[i],m[j]$)\;
      \Add($w,i$)\;
      \ForEach{$k\in\{i,j\}$}{
        $d[k]\leftarrow$ \MinimizePath($n, k,m[k]$)-$k$\;
      }
      \Finished$\leftarrow$ false\;
      break\;
    }
  }
}
return $w$\;
\caption{Sorting by uncrossings. The algorithm returns a geodesic between $m$ and $m'$, where $m'$ is sorted in the sense that all points have their corresponding neighbours, but the whole configuration may be rotated and/or flipped by a dihedral action.}
\label{alg:sort}
\end{algorithm}

\subsection{Application to \emph{Yersinia} genomes}

We apply the method summarised in Algorithm~\ref{alg:length} to calculate inversion-based distances among eight \emph{Yersinia} genomes. The input data are in the form of a set of permutations of regions that are conserved across all genomes. We obtained these permutations from~\citet{Darling2008} by using the Mauve software package~\cite{Darling2010progressivemauve}. The resulting matrix of minimal inversion distances is given in Table~\ref{tab:matrix}. 
This matrix of distances can be used to generate a phylogenetic tree using distance-based methods such as neighbour-joining~\cite{Saitou1987neighbour}. We applied this method using the phylip package~\cite{Felsenstein1989phylip}. The two genomes of {\it Yersinia pseudotuberculosis} can be used as outgroups, as done for example by~\citet{bos2011draft}.  The resulting phylogeny is shown in Figure~\ref{fig:tree}.

\begin{table}
  \caption{Matrix of minimal inversion distances among \emph{Yersinia} genomes calculated by Algorithm~\ref{alg:length}. \label{tab:matrix}}
\begin{center}
\begin{tabular}{lrrrrrrrr}
           &  \begin{sideways}KIM\end{sideways}  &\begin{sideways}ANTIQUA\end{sideways}  & \begin{sideways}MICROTUS\end{sideways}&\begin{sideways}CO92\end{sideways} &\begin{sideways}NEPAL516\end{sideways}  &\begin{sideways}PESTOIDES\end{sideways}  &\begin{sideways}Yp\_IP31758\end{sideways}  &\begin{sideways}Yp\_IP32953\end{sideways}  \\
\hline
  KIM         &  0  &233  & 738&188  &334  &515  &758  &738  \\
  ANTIQUA     &233  &  0  & 750&319  &449  &664  &719  &712   \\
  MICROTUS    &738  &750  & 0&745  &659  &809  &695  &706   \\
  CO92        &188  &319  & 745 &  0  &366  &595  &697  &760   \\
  NEPAL516    &334  &449  & 659&366  &  0  &659  &641  &759   \\
  PESTOIDES   &515  &664  & 809&595  &659  &  0  &753  &695   \\
  Yp\_IP31758 &758  &719  & 695&697  &641  &753  &  0  &589   \\
  Yp\_IP32953 &738  &712  & 706&760  &759  &695  &589  &  0   \\ \hline
\end{tabular}

\end{center}

\end{table}

This evolutionary reconstruction can be compared to the results of~\citet{Darling2008}, who also used inversion information. While Darling et al.\ used a network visualisation of the relationships among the genomes, it is possible to see the similarities with our phylogeny. 
Namely, \emph{Y. pestis}~Pestoides and \emph{Y. pestis}~Microtus 91001 join near the root, the remaining  \emph{Y. pestis} isolates group together, and the two {\it Yersinia pseudotuberculosis} outgroups also group together. 
Another point of comparison is the phylogeny based on 1,694 variable positions across the whole genome~\cite{bos2011draft}.  Again, Pestoides and Microtus 91001 diverged early while the remaining genomes evolved more recently. Bos et al distinguish two clades that arose since the Black Death: one with Nepal516 and KIM and the other with CO92 and Antiqua. Although there are slight differences within these recent clades, our methods produced a tree that is broadly consistent with the tree of~\citet{bos2011draft} which uses sequence variation -- a completely different source of information.

\begin{figure}[ht!]
\begin{center}
\includegraphics[width=.6\textwidth]{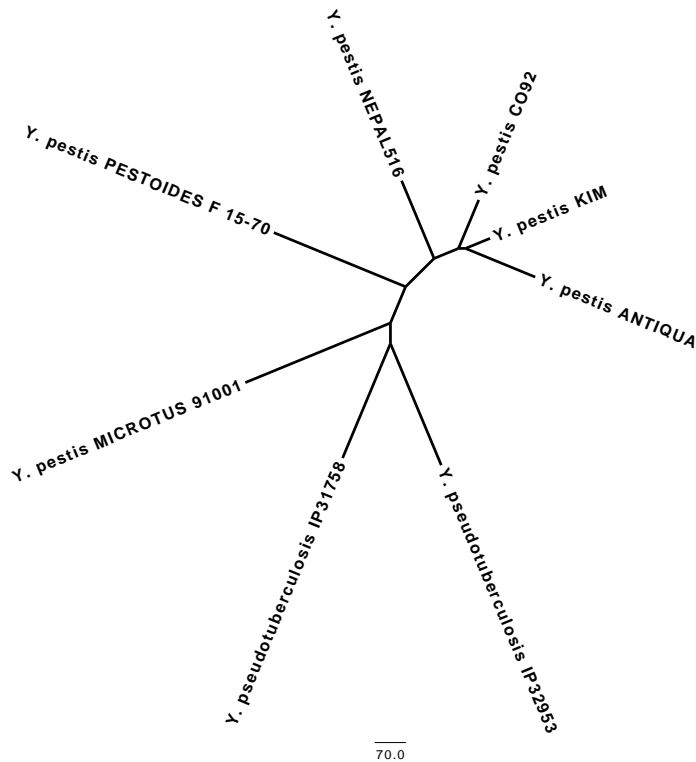}
\end{center}
\caption{Phylogeny from data published in~\citet{Darling2008}, based on distances obtained by applying Algorithm~\ref{alg:length}. \label{fig:tree}}
\end{figure}

\section{A general modelling framework}

We have studied a model in which only certain inversions are permitted, specifically those of two adjacent regions.  As remarked in Section~\ref{s:systems}, this is but one example of an \emph{inversion system}, in which the set of inversions $\inv$ is constrained in some way.  In this setting, we define a metric $\ell$ on the group relative to $\inv$ and according to parsimony, so that $\ell(g)$ is the word length of $g$ in the generators $\inv$.  Then  
the distance between genomes $G_1$ and $G_2$, represented by group elements $g_1$ and $g_2$ in the model $(G,\inv)$ with the metric $\ell$ is simply $\ell(g_1^{-1}g_2)$.  The model in which all inversions are permitted, the \emph{uniform model}, simply removes constraints on $\inv$ completely.

A more realistic model than either the uniform or a constrained model is one in which inversions may be given different weights according to experimental data, such as the frequencies of different inversion lengths.  Estimates of these frequencies have been obtained for instance in~\citet{Darling2008} (see Figure~7 of that paper).  In the group-theoretic setting, we have the same group and the same generators as in the uniform model, since all inversions are permitted.  The variation in the frequencies of different inversions is accounted for by manipulating the \emph{length function} in the group, for instance by assigning weights to inversions depending on the number of regions involved~\cite{swidan2004sorting}.  In the uniform and constrained models, and in most combinatorial group theory, each generator is defined to have length 1, and the length of a product of generators is the sum of the lengths.  However, length is used as a proxy for evolutionary distance, and if inversions are not equally probable then their length should be different.  In the light of the parsimony assumption that the most likely evolutionary path is one of minimal distance, the weighting (or length) of a single inversion needs to be adjusted to account for the difference in probability.  

For example, let $\inv$ be the set of all possible inversions, and suppose $\omega:\inv\to\R^{\ge 0}$ is a weight function assigned to the inversions.  If we assume short inversions are more probable than longer inversions, we may have a weight function that is order-preserving with respect to inversion length (if $s$ is longer than $t$ then $\omega(s)>\omega(t)$).  Then our metric on the group may be defined by first defining length additively on any word in the generators, setting $\ell(s_{i_1}\dots s_{i_k})=\omega(s_{i_1})+\dots+\omega(s_{i_k})$.  Then to define the length of a group element $g\in G$ one needs to take the minimum over all words $w$ in the generators $\inv$ representing $g$: $\ell(g)=\min\{\ell(w)\mid w=g\}$.  This establishes the minimal weight path in the Cayley graph from the identity arrangement to $g$. In practice this is a significant problem, however, as there are infinitely many words representing $g$.  Even eliminating paths that double back on themselves the search space is potentially very large.  Applying this model in generality will require some clever new ideas or a statistical approach (some computational approaches have been taken in~\citet{pinter2002sorting,swidan2004sorting}).

While the general version of this model seems difficult to work with, special cases are clearly not intractable, as we have shown in this paper.  The two-region inversion model we study simply employs a special weight function in which $\omega(s)=1$ if $s$ is an inversion of two regions and 0 otherwise.  Indeed, any model that restricts the set of inversions $\inv$ but treats them all as equivalent follows a similar pattern, setting $\omega(s)=1$ if $s\in\inv$ and 0 otherwise.  This also applies to models in which inversions are not restricted according to length but by location, such as models only permitting inversions that do not move the terminus of replication, or even that are symmetric about the terminus~\cite{ohlebusch2005median}.   Similarly the uniform model generally studied has the even simpler weight function $\omega(s)=1$ for all inversions $s$.

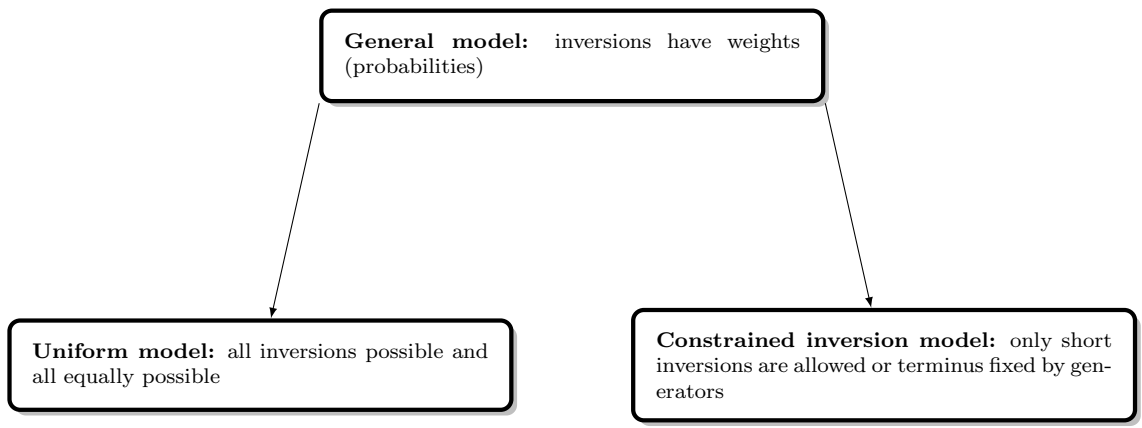
\begin{figure}[ht!]
\begin{center}
\begin{tikzpicture}
\tikzstyle{rec} = [rectangle,rounded corners,ultra thick,draw=black,fill=white,node distance=5.8cm] 

\node[rec,inner sep=3mm,drop shadow] (prob)
	{\parbox{0.4\textwidth}{\footnotesize\textbf{General model:} inversions have weights (probabilities)}};

\node[rec,inner sep=3mm,drop shadow,below left of=prob] (curr)
	{\parbox{0.4\textwidth}{\footnotesize\textbf{Uniform model:} all inversions possible and all equally possible}};

\node[rec,inner sep=3mm,drop shadow,below right of=prob] (fix)
	{\parbox{0.4\textwidth}{\footnotesize\textbf{Constrained inversion model:} only short inversions are allowed or terminus fixed by generators}};

\draw[->,>=latex]  (prob.south west) -- (curr);
\draw[->,>=latex]  (prob.south east) -- (fix);

\end{tikzpicture}
\end{center}

\caption{Different biological models. The widely used uniform model and any constrained model are special cases of the general model, in which all inversions are assigned a weight value. The weights can be interpreted as probabilities  the inversions to occur.}
\label{fig:models}
\end{figure}

\section{Discussion}

In this article, we have introduced a group theoretic framework for modelling the process of bacterial genome rearrangement due to inversions. Using this framework, we outlined a range of alternative models. We focused on a specific model in which inversions act locally on two genomic regions at a time.  Based on this model, the group theoretic framework has enabled us to derive a new algorithm to obtain the minimum number of inversion events connecting two genomes under comparison.  The key conceptual step has been to find a way to lift circular permutations to the affine symmetric group in such a way that the inversion distance on the circular genome can be found using results on length of elements in the affine symmetric group.  

The combinatorial group theory of permutation groups has a long history of development and therefore presents a potent opportunity to examine in a new light the processes underlying bacterial genome evolution. 
There is potential to introduce more realism into models of genome evolution by generalising the model studied here within the group
theoretic framework.
This represents an important advance over existing methods of comparative genomics based on fairly coarse assumptions, most particularly the assumption that all inversion events are equally probable.  
In addition, while the questions in group theory that arise in the algebraic models in this paper are new, they are related to other questions under ongoing consideration by mathematicians. It is to be hoped that this connection between evolutionary biology and algebra will drive further theoretical development of related group theory.

Our approach extends preceding studies that applied group theory to genome evolution. While the innovative study by~\citet{Watterson-chrom-reversal-1982} described the inversion distance problem, its interpretation as a problem in group theory was first noted by~\citet{kececioglu1993exact} a decade later (and followed by~\citet{meidanis2000alternative}).  These models, along with most other approaches to the problem, assume a uniform distribution of inversion lengths; something we have addressed in this paper by allowing only inversions of two adjacent regions.  A wider issue is that of whether the minimal length is the best measure of evolutionary distance at all, given evolution may not have taken a shortest path to the observed arrangements, regardless of the metric used to define minimal.
Recently,~\citet{moulton2011butterfly} pursued this challenge to parsimony, using group theoretic principles to consider the effect on length of a small change to an inversion sequence, and obtained ``worst-case'' bounds on the difference between lengths of elements when an additional generator is used.  In general, it is clear that the application of group-theoretic methods to genomics problems holds great promise for a fertile exchange between algebraists and evolutionary biologists.

\bibliographystyle{apalike}
\bibliography{../inversions}

\end{document}